\documentclass[ECP]{ejpecp} 

\SHORTTITLE{Chase-escape on the configuration model}

\TITLE{Chase-escape on the configuration model} 

\AUTHORS{%
  Emma Bernstein\footnote{Bard College
    \EMAIL{emm.bernstein@gmail.com}}
  \and 
  Clare Hamblen\footnote{University of Oregon \BEMAIL{chamble2@uoregon.edu}}
  \and 
  Matthew Junge \footnote{Baruch College \BEMAIL{matthew.junge@baruch.cuny.edu}}
  \and 
  Lily Reeves \footnote{Cornell University \BEMAIL{zw477@cornell.edu}}
  
  }


\AMSSUBJ{60K35} 

\SUBMITTED{December 17, 2021} 
\ACCEPTED{0000} 

\VOLUME{0}
\YEAR{2020}
\PAPERNUM{0}
\DOI{10.1214/YY-TN}

\ABSTRACT{Chase-escape is a competitive growth process in which red particles spread to adjacent empty sites according to a rate-$\lambda$ Poisson process while being chased and consumed by blue particles  according to a rate-$1$ Poisson process. Given a growing sequence of finite graphs, the critical rate $\lambda_c$ is the largest value of $\lambda$ for which red fails to reach a positive fraction of the vertices with high probability. We provide a conjecturally sharp lower bound and an implicit upper bound on $\lambda_c$ for supercritical random graphs sampled from the configuration model with independent and identically distributed degrees with finite second moment. We additionally show that the expected number of sites occupied by red undergoes a phase transition and identify the location of this transition. }

\usepackage{theoremref, mathtools, comment}
\mathtoolsset{showonlyrefs}
\newcommand{\f}{\frac}
\newcommand{\PP}{\mathbb{P}}
\newcommand{\EE}{\mathcal{E}}
\renewcommand{\P}{\mathbf{P}}
\DeclareMathOperator{\Bin}{Bin}
\newcommand{\0}{\rho}
\renewcommand{\b}{\mathfrak b}

\newcommand{\T}{\mathbb{T}}
\newcommand{\ind}[1]{\mathbf{1}{\{ #1 \}}}
\newcommand{\HOX}[1]{\marginpar{\footnotesize #1}}

\newcommand{\E}{\mathbf E}

\newcommand{\R}{\mathcal R}
\renewcommand{\S}{\mathcal S}
\newcommand{\M}{\mathcal M}
\newcommand{\G}{\mathcal G}
\newcommand{\HH}{\mathcal H}
\newcommand{\D}{\mathcal D}

\begin{document}



\section{Introduction}

Parasite-host models in which parasite expansion is restricted to the sites occupied by their hosts were introduced by ecologists Keeling, Rand, and Wilson \cite{keeling1995ecology, rand1995invasion}. The same dynamics have been reinterpreted in a variety of applications: predator-prey systems, rumor scotching, infection spread, and malware repair in a device network \cite{bordenave2014extinction, rumor, de2015process, gilbert}. Despite  it being natural to model these applications on finite networks, there are few rigorous results to this end. The goal of this present work is to prove that there is a phase transition for species proliferation on sparse random graphs generated from the configuration model. 

We begin by defining \emph{chase-escape} in general. Let $G'$ be a graph with root $\0$. We obtain $G$ by attaching an additional vertex $\b$ to $\0$. Vertices of $G$ are in states $\{r,b,w\}$.  Pairs of adjacent vertices in states $(b,r)$ transition to $(b,b)$ according to a rate-$1$ Poisson process. Pairs of adjacent vertices in states $(r,w)$ transition to $(r,r)$ according to a rate-$\lambda$ Poisson process with $\lambda>0$. 
The initial configuration has $\0$ in state $r$, $\b$ in state $b$, and all other vertices in state $w$. 
We will interchangeably refer to the vertices in states $r,b,$ and $w$ as either red, blue, and white vertices, or as being occupied by red, blue or white particles. One interpretation of these dynamics is of red (prey) being chased and consumed by blue (predators). It is from this interpretation that the chase-escape process earns its name.

The process {fixates} if at some time there are no more red vertices. Let $\R$ be the set of sites that are at some time red and $| \R |$ denote the number of vertices in $\R$.
When $G$ is an infinite graph, we define the critical value
\begin{align}
    \lambda_c(G) = \sup \{ \lambda \colon \P_\lambda(|\R| = \infty) = 0 \} \label{eq:lc}
\end{align}
as the fastest spreading rate at which red almost surely occupies only finitely many vertices of $G$. 

Increasing $\lambda$ allows for more red expansion, but more expansion makes more sites available for blue to chase red along. These offsetting factors make it difficult to couple systems with different values of $\lambda$. For example, for any non-trivial graph with cycles, it is not known that $\P_\lambda(|\R|= \infty)$ increases in $\lambda$. 

A standing conjecture, informally attributed to Martin, is that $\lambda_c(\mathbb Z^d)<1$ for $d \geq 2$.  This was supported by simulation evidence from Tang, Kordzakhia, and Laller who predicted that $\lambda_c(\mathbb Z^2) \approx 1/2$ \cite{tang2018phase}. A more thorough simulation study from Kumar, Grassberger, and Dhar provided convincing evidence that $\lambda_c(\mathbb Z^2) =0.49451 \pm 0.00001$ \cite{kumar2021chase}. The authors gave further quantitative evidence which suggested that the critical exponents for chase-escape fall in the same universality class as those of undirected bond percolation. 

Kordzakhia proved in \cite{tree1} that, when $G = \T_d$ is the infinite rooted $d$-ary tree in which each vertex has $d$ child vertices, we have
\begin{align}
\lambda_c(\mathbb T_d) = 2d -1 - 2\sqrt{d(d-1)} \sim \f 1 {4d}.\label{eq:tree}\end{align}
This comes from balancing the exponential growth of paths in the trees against the large-deviation event that red is able to survive along a fixed path (see \cite{DJT} for a simple argument). 
Note that \eqref{eq:tree} is strictly less than $1$ for $d>1$. Later, Bordenave generalized the formula at \eqref{eq:tree} to trees with arbitrary branching number  \cite{bordenave2014extinction}, which was further refined by Kortchemski \cite{tree_chase}. 

Beckman, Cook, Eikmeier, Junge, and Hernandez-Torres studied a variant of chase-escape on $\mathbb T_d$ in which red particles die at some rate \cite{ced}. Durrett, Junge, and Tang in \cite{DJT} proved that $\lambda_c(\mathbb T_d)\leq \lambda_c(G)$ for any graph $G$ with no more than $d^n$  self-avoiding paths of length $n$ starting from the root. They also proved that $\lambda_c(G) = 1$ when $G$ is the ladder graph $\mathbb Z \times \{ 0,1\}$ and that red may reach infinitely sites even when it spreads at a slower rate than blue in a modified version of chase-escape on the oriented two-dimensional lattice with spreading rates that resemble Bernoulli bond percolation. 
Interpreting chase-escape dynamics as malware spread/repair in a device network, Hinsen, Jahnel, Cali, and Wary demonstrated a phase transition for coexistence on Gilbert graphs \cite{gilbert} with a positive density of vertices blue at the onset. 

We study chase-escape on sequences of graphs $\G = (G_n)_{n \geq 1}$ for which $G_n$ has vertex set $\{\b, 1, \hdots, n\}$. We identify $1$ as the root and assume one edge is present between $\b$ and $1$. Let $\R_n$ be the set of sites occupied by red at some time in chase-escape on $G_n$ with $\b$ initially blue and $1$ initially red.  Define \begin{align}
    \lambda_c(\G) = \sup \{ \lambda \colon \inf_{\delta >0} \limsup_{n \to \infty} \P_\lambda(|\R_n| > \delta n ) = 0\}. \label{eq:lcf}
\end{align}
In words, $\lambda_c(\G)$ is the largest value such that, if $\lambda < \lambda_c(\G)$, the probability that red occupies a fixed fraction of the vertices of infinitely many of the $G_n$ goes to zero. Since $\P_\lambda( |\R_n| > \delta n)$ is not known to be monotonic in $\lambda$, it is an open problem to prove that the definition at \eqref{eq:lcf} is equivalent to the infimum formulation
\begin{align}
\lambda_c'(\G) = \inf \{ \lambda \colon \exists \delta, \epsilon >0 \text{ such that } \limsup_{n \to \infty} \P_\lambda(|\R_n| > \delta n) > \epsilon\} \label{eq:lcf'}.
\end{align} 
Nonetheless, by definition, we have $\lambda_c(\G) \leq \lambda_c'(\G)$. 

 The scaling limit of chase-escape on the complete graph, known as the birth-and-assassination process \cite{aldous}, was studied by Bordenave in \cite{rumor}. Kortchemski later showed that the number of remaining white sites when the process fixates is $O(n^{1- \lambda})$ for $\lambda \in (0,1)$ and $O(1)$ for $\lambda\geq 1$. Thus, red reaches most, if not all, of the vertices of the complete graph for any $\lambda >0$ \cite{complete}. Arruda, Lebensztayn, Rodrigues, and Rodr\'iguez provided simulation evidence in support of the conjecture that similar behavior occurs when $G_n$ is a dense Erd\H{o}s-R\'enyi graph \cite{de2015process}. 
 
 The \emph{two-type Richardson growth model} is closely related to chase-escape. In this process red and blue compete for territory on a graph with blue spreading at rate-1 and red at rate-$\lambda$. Both colors spread only to white vertices and, upon reaching a white site, occupy that site for all subsequent time. The process has implications for geodesics in first passage percolation on $\mathbb Z^d$ \cite{hoffman, pemantle}. Antunovi\'c, Dekel,  Mossel, and Peres studied this process on random regular graphs and proved that both red and blue occupy a positive fraction of the vertices with positive probability only when $\lambda =1$. Deijfen and Van der Hofstad proved that one species dominates for the process on the configuration model with $P(D > x) \approx x^{-\tau + 1}$ with $\tau \in (2,3)$ \cite{winner}. There has been a recent study of a somewhat similar competing two-type growth processes with conversion by Finn and Stauffer \cite{finn2020nonequilibrium} as well as a result by Candellero and Stauffer \cite{candellero2021passage} that a related process lacks a seemingly intuitive monotonicity property.

\subsection{Results}

 We are interested in $\lambda_c(\G)$ when $\G$ consists of random (multi-)graphs sampled from the configuration model with an independent and identically distributed degree sequence. Let $D$ be a random variable supported on the nonnegative integers and define $\mathcal D = (D_1,D_2,\hdots)$ to be a sequence of independent random variables with the same distribution as $D$. We sample a graph $G_n$ on the vertices $1,\hdots, n$ by first assigning $D_i$ half-edges to each vertex $i$ ($\b$ receives no half-edges). 
 If $\sum_{i=1}^n D_i$ is odd, then we add a half-edge to the vertex $n$ so that there is an even number of half-edges.  A \emph{matching}, that pairs each half-edge with another distinct half-edge and only uses each half-edge once, is chosen uniformly at random from the set of all such matchings. $G_n$ is obtained by including an edge for each pair of matched half-edges. Note that this construction allows for the possibility of multiple- and self-edges in $G_n$.
 We identify the root $\0$ with the vertex $1$ and point out that $\deg 1 = D_1 + 1$ since the blue vertex $\b$ is attached. The presence of the extra edge at $1$ and possibly at $n$ are minor nuisances that have no significant impact as $n \to \infty$.

 Denote the joint measure for chase-escape on $\G$ by $\P_\lambda(\cdot)$ and the expectation with respect to this measure by $\E_\lambda[\cdot]$. For the probability measure and expectation of generic random variables such as the degrees $D, D_1,D_2,\hdots$ we use $P(\cdot)$ and $E[\cdot]$.  We will say that a sequence of events occurs \emph{with high probability} if the probability converges to 1. 


To have a phase transition for chase-escape on $\G$, the graphs $G_n$ must contain giant components, i.e., connected components with a non-vanishing fraction of the vertices. To this end, we assume that
\begin{align}
0< E[D (D -2)] < \infty \label{eq:mr} .
\end{align}
This is a variant of the Molloy-Reed criterion for the emergence of a giant component and was first observed in  \cite{molloy1995critical}. Even with a giant component, proving that red can reach a positive fraction of that component remains a complicated question. That red is able to do so for large enough $\lambda$ and fails for small $\lambda$ is our first result.

Before stating it, we introduce two important quantities:   
\begin{align}\label{eq:Lambda}
a &= (E [D^2] /E [D]) -1\\
\Lambda&:= \Lambda(a) = 2a -1 - 2 \sqrt{ a^2 - a} \sim \f 1 { 4a}
\end{align}
with the convention that $\Lambda(\infty) = 0$. 
Graphs sampled from the configuration model are known to be tree-like with branching number $a$ \cite{bollobas2015old}.
So, the form of $\Lambda$ at \eqref{eq:Lambda} matches that of the critical value for infinite trees with a given branching number at \eqref{eq:tree} from \cite{bordenave2014extinction}. For this reason, we make the following conjecture.
\begin{conjecture} \thlabel{conj:lc}
Let $\lambda_c(\G)$ be as defined at \eqref{eq:lcf} and $\Lambda$ as defined at \eqref{eq:Lambda}. Suppose that $\G$ is sampled from the configuration model with an independent $D$-distributed degree sequence satisfying \eqref{eq:mr}. We conjecture that $\lambda_c(\G) = \Lambda$.
\end{conjecture}
This is discussed more in Section \ref{sec:future}. Our first result is a lower bound consistent with \thref{conj:lc} and a proof that $\lambda_c(\G)$ is finite.

\begin{theorem} \thlabel{thm:main}
For chase-escape with the same conditions as \thref{conj:lc} it holds that
$$\Lambda  \leq \lambda_c(\mathcal G) < \infty.$$ 
\end{theorem}

 Our second result is that a phase transition for the expected number of sites reached by red occurs at $\Lambda$. 
 
\begin{theorem}\thlabel{thm:ER}
Consider chase-escape with the same conditions as \thref{conj:lc}. If $\lambda \leq \Lambda$, then there exists a constant $C:=C(\Lambda)$ independent of $\lambda$ such that
\begin{align}
\sup_{n \geq 1} \E_\lambda [|\R_n|] < C.\label{eq:finE}
\end{align}
If $\lambda >\Lambda$, then 
\begin{align}\lim_{n \to \infty} \E_\lambda [|\R_n|] = \infty.\label{eq:ERl}
\end{align}
If $E[D^2] = \infty$ and $E[D]<\infty$, then this result continues to hold with $\Lambda = 0$. 
\end{theorem}

Notice that the phase transition is discontinuous; $\E_\lambda [ |\R_n|]$ is uniformly bounded for all $\lambda \leq \Lambda$, and diverges for larger $\lambda$. Similar behavior for the expected range was  observed for chase-escape on trees in \cite{tree_chase}. Interestingly, the analogous phase transition was proven to be continuous for chase-escape with death \cite{ced}.

\subsection{Proof overview}
The lower bound on $\lambda_c(\G)$ is a relatively straightforward extension of the ideas in \cite[Corollary 1.2]{DJT} which uses the fact that the configuration model is locally tree-like. We prove \thref{thm:ER} using the tree-like local structure of $G_n$ along with a result from \cite{bordenave2014extinction} concerning survival of chase-escape on general trees.

The proof of the upper bound on $\lambda_c(\G)$ is inspired by the approach from \cite{gilbert}. 
We restrict to \emph{open} vertices from which red spreads to all of the neighbors before any adjacent blue vertex overtakes the red particle there. Red spreads uninhibited between such sites.
Our argument comes down to proving that the subgraph of open vertices and the edges between them contains a giant component with non-vanishing probability. 
Establishing this is subtle because high-degree vertices, which intuitively are the most crucial for making the component well-connected, are less likely to be included. 
While there is an extensive treatment of degree-dependent percolation in the applied literature under the name ``network robustness," we were unable to find many rigorous results in this area \cite{callaway2000network}. Our argument makes use of a recent criterion for a  uniformly random simple graph with a given degree distribution to contain a giant component from Joos, Perarnau, Rautenbach, and Reed \cite{JPRR18}. This amounts to showing that for $\lambda$ large enough, not too many vertices are deleted. So if $G_n$ contains a giant component, then the subgraph is likely to as well. 

\subsection{Future work} \label{sec:future}
The most natural next step is proving \thref{conj:lc}. While we do not have a concrete bound when proving $\lambda_c(\mathcal G) < \infty$, the implicit bound in \thref{thm:main} diverges as $D$ becomes stochastically larger.  Since $\Lambda \to 0$ as $D$ becomes large, it would be great to prove a universal upper bound for $\lambda_c(\G)$ that tends to zero for sparse random graphs. This would be a worthy pursuit even for special cases such as sparse Erd\H{o}s-R\'enyi or random regular graphs. 
 Our approach does not appear to extend to dense graphs where it is natural to conjecture that $\lambda_c =0$. However, to our knowledge, it is an open problem to prove that red reaches a positive fraction of vertices on such graphs for any $\lambda >0$. For example, it would be nice to prove this statement on Erd\H{o}s-R\'enyi graphs $G(n,p_n)$ with $\lim_{n \to \infty} n p_n \to \infty$. We believe that the value of $\lambda_c$ is unknown for this case even when $p_n \equiv p$, although there have been simulation studies \cite{de2015process} suggesting this is true.

\subsection{Organization}

In Section \ref{sec:upperbd} we prove the upper bound on $\lambda_c$ from \thref{thm:main}. In Section \ref{sec:ER} we prove upper and lower bounds on $\E_\lambda[|\R_n|]$ to complete the proof of \thref{thm:main} and proof of \thref{thm:ER}. 

\section{The upper bound}  \label{sec:upperbd}

We begin by reducing the problem of proving $\lambda_c(\G) < \infty$ to a question about degree-dependent percolation. As described in more detail in \cite{DJT}, the Markov property of the underlying Poisson processes allows us to represent chase-escape as a collection of independent red and blue passage times along each edge. Fix a vertex $v$. We imagine that $G_n$ has not been sampled yet so that $v$ has $D_v$ half-edges that will eventually be connected to other vertices in the graph. To each half-edge, we assign outbound red passage times $(\tau_{v,i}^r)_{i=1}^{D_v}$ and inbound blue passage times $(\tau_{v,i}^b)_{i=1}^{D_v}$. The $\tau_{v,i}^r$ are independent exponential random variables with mean $1/\lambda$ and the $\tau_{v,i}^b$ are independent exponential random variables with mean $1$.

Let $T_v^r = \max_{1 \leq i \leq D_v} \tau_{v,i}^r$ and $T_v^b =  \min_{1 \leq i \leq D_v} \tau_{v,i}^b.$
We call $v$ \emph{open} if the event
\begin{align}
O_v = \left\{ T_v^r < T_v^b\right \} \label{eq:open}
\end{align}
occurs.
In words, if $v$ becomes red, then the red particle at $v$ spreads to all neighbors of $v$ before $v$ can be overtaken by a blue particle at any of its neighbors. Note that $O_1$ is slightly different from the other $O_v$ since the presence of $\b$ ensures that $\deg 1 = D_1 +1$. Thus, we must include the additional passage time $\tau_{1,\b}^b$ for the edge $(\b,1)$. Similarly, we may need to include additional passage times at $n$ if the sum of the degrees is odd.  


Let $H_n$ be the subgraph of $G_n - \{\b\}$ consisting of all open vertices and edges between open vertices. Denote the largest connected component in $H_n$ by $H_n^{(1)}$. We say that the sequence $\HH := (H_n)_{n \geq 1}$ \emph{contains a giant component} if there exists $\delta >0$ such that
\begin{align}
\limsup_{n \to \infty} \P_\lambda(|H_n^{(1)}| >\delta n)>0. \label{eq:Hgiant}
\end{align}

\begin{lemma}\thlabel{lem:giant}
If $\lambda$ is such that $\HH$ satisfies \eqref{eq:Hgiant} for some $\delta >0$, then $\lambda_c(\G) < \lambda.$
\end{lemma}

\begin{proof}
 Disregarding the edge $(\b,1)$,  the root is indistinguishable from the other vertices. This self-similarity ensures that 
 \begin{align}\P_\lambda(1 \in H_n^{(1)} \mid |H_n^{(1)}| \geq k) \geq k/n. \label{eq:Hcond}
 \end{align}
Moreover, if $1 \in H_n^{(1)}$, it follows from the definition of an open vertex that all vertices in $H_n^{(1)}$ are at some time red in chase-escape on $G_n$. Thus, the two random variables are coupled so that $|\mathcal R_n| \geq |H_n^{(1)}|$ almost surely.
Conditioning on the event $|H_n^{(1)}| \geq \delta n$ then applying \eqref{eq:Hcond}, we may write
$$\P_\lambda (|\mathcal R_n| \geq \delta n) \geq \P_\lambda(|H_n^{(1)}| \geq \delta n \text{ and } 1 \in H_n^{(1)}) \geq \delta \P_\lambda (|H_n^{(1)}| \geq \delta n).$$ 
So, whenever \eqref{eq:Hgiant} holds, we have 
\begin{align}
    \limsup_{n \to \infty} \P_\lambda(|\R_n| > \delta n) \geq \delta \big[ \limsup_{n \to \infty} \P_\lambda(|H_n^{(1)}| >\delta n)\big] >0. \label{eq:suf}
\end{align}
We then have $\lambda_c(\G) \leq \lambda$ by definition. 
\end{proof}
In light of \thref{lem:giant}, the upper bound on $\lambda_c(\G)$ from \thref{thm:main} follows from finding $\lambda$ and $\delta$ satisfying \eqref{eq:Hgiant}. 

\begin{remark} \cite{callaway2000network, newman2003structure} provided a heuristic that $\HH$ contains a giant component with positive probability so long as 
$$\sum_{k>0} k (k-1) \P_\lambda(O_v \mid \deg v = k) P(D = k) > E D.$$
Since $\P_\lambda(O_v \mid \deg v = k) \to 1$ as $\lambda \to \infty$, this condition holds for large enough $\lambda$ by our assumption at \eqref{eq:mr}. Unfortunately, the heuristic does not account for the dependence in the degree sequence for $\HH$. We must take a longer route to obtain the desired result. 
\end{remark}

The first step is observing that $H_n$ has the law of a configuration model on a modified degree sequence.

\begin{lemma} \thlabel{lem:law}
Let $\hat{\mathcal D}_n = (\hat D_1,\hdots, \hat D_n)$ be a sequence of degrees with 
$$\hat D_i = \begin{cases}  
\text{$\deg(i)$ in the subgraph $H_n$}, & \text{ if } i \in H_n \\
0, & \text{ if } i \notin H_n
\end{cases}$$
Note that the edges from vertices in $H_n$ to those in $H_n^c$ are not counted in $\hat D_i$. 
We claim that the law of $H_n$ is that of a graph sampled from the configuration model with degree sequence $\hat \D_n$.
\end{lemma}

\begin{proof}
 Call a vertex \emph{closed} if it is not open. The open/closed status of each vertex is locally determined by the passage times at the site. Thus, whether or not each vertex is open is independent of the edge configuration of $G_n$. We begin by sampling the passage times and labeling the vertices of $G_n$ as open or closed. We then reveal the edge connections for only the closed vertices in $u \in H_n^c$. After revealing this, the closed vertices have degree $\hat D_u = 0$. The open vertices $v \in H_n$ have degree $\hat D_v$ equal to $D_v$ minus any half-edges that were revealed to be connected to closed vertices. The remaining half-edges correspond to the connections in $H_n$ which are yet to be revealed. The matching between these edges is independent of the exploration of $G_n$ to remove closed edges. Generating $H_n$ is thus a configuration model with degree distribution $\hat{\mathcal D}_n$. 
\end{proof}

A downside of the previous construction is that the vertex degrees in $\hat \D_n$ are dependent. Indeed, we know that with high probability $\sum_1^n (D_i - \hat D_i)$ is on the order of $n$ which imposes some structure on the $\hat D_i$. There are many theorems giving criteria for the emergence of giant components in graphs from the configuration model \cite{bollobas2015old, molloy1995critical, janson2009new}. However, these theorems assume that the empirical degree counts $m(k) = |\{i \leq n \colon \hat D_i = k\}|$ converge. While this is likely true for $\hat \D_n$, it is not so obvious how to deal with the dependence. We choose to proceed in a different manner that uses a recent, rather robust condition for giant components to exist in uniformly random simple graphs taken from \cite{JPRR18}. We can apply this condition since the law of the configuration model, conditioned to be simple, and the law of uniformly sampling a simple graph with the same degree sequence are the same, and, when $D$ satisfies \eqref{eq:mr}, the graphs $G_n$ (and thus $H_n$) are simple with positive probability uniformly in $n$ \cite{janson2009probability}.  

We define a few important quantities. Let $\mathcal D_n = (D_1, \dots, D_n)$ be a general, possibly dependent, degree sequence. Arrange $\mathcal{D}_n$ in increasing order $D_{\pi_1} \leq \hdots \leq D_{\pi_n}$
and define the quantities
\begin{align} \label{eq:mr-def}
    j_n &= \min \Big( \{j \leq n \colon \sum_{i=1}^j D_{\pi_i}(D_{\pi_i}-2) >0 \} \cup \{n \} \Big), \\
    \S_n &= \sum_{i=j_n}^n D_{\pi_i}, \quad 
    \mathcal{M}_n = \sum_{D_i \neq 2} D_i, \text{ and } \\
    \EE_n &= \sum_{i\in H_n^c} D_{i}.
\end{align}

\begin{lemma} \thlabel{lem:giant} 

 If $\S_n \ge \epsilon \M_n$ from \eqref{eq:mr-def} with high probability for some fixed $\epsilon >0$, then there exists $\delta>0$ such that $G_n$ has a component containing at least $\delta n$ vertices with high probability.
\end{lemma}
\begin{proof}
This is a restatement of \cite[Theorem 2]{JPRR18} for the degree distributions we consider.
\end{proof}

We now would like to show that the criterion in \thref{lem:giant} can be applied to $\hat \D_n$ to infer that $\HH$ contains a giant component. We do so by showing that $\G$ contains a giant component, and then show that the ``damage" done by removing closed edges does not impact the criterion at \eqref{eq:mr-def} in a serious way. This is done formally by controlling the value of $j_n$, which can be thought of as the minimum amount of $\D_n$ that must be revealed for a Molloy-Reed-type condition to hold. 

To start we need to control the left and right tails of summing order statistics as in the definitions of $j_n$ and $\mathcal S_n$.

\begin{lemma} \thlabel{lem:lln}
Let $X_1, X_2,\hdots$ be independent and identically distributed random variables supported on the integers with finite mean $\mu>0$. Denote the $i$th order statistic of the sub-collection $X_1,\hdots, X_n$ by $X_{\pi_i}^n$. Given $\epsilon \in (0,1)$, there exists $\delta >0$ such that for $n_\delta = \lceil (1- \delta) n \rceil$ we have
\begin{align}
 \f 1n\sum_{i=1}^{n_\delta } X_{\pi_i}^n \geq (1-\epsilon)\mu
\text{ and  }
 \f 1n \sum_{i=n_\delta +1}^{n} X_{\pi_i}^n \leq \epsilon \mu  \text{ with high probability.} \label{eq:lln}
\end{align}
\end{lemma}

\begin{proof}
First we assume that $P(X_1 > x) >0$ for all $x\geq 0$ so that $X_1$ has unbounded positive support. Set $$M := \min \{ x \colon E [X_1 \ind{X_1 \leq x}] \geq (1- \epsilon/3) \mu  \},$$ which exists by the hypothesis that $E [X_1] = \mu$. Let $0< p := P(X_1 > M)$. 
The law of large numbers ensures that 
\begin{align}\f 1n \sum_1^n \ind{X_i \leq M} \leq 1- \f p 2 >0 \text{ and } \f 1n \sum_1^n X_i \ind{X_i \leq M} \geq (1- \epsilon/2) \mu \label{eq:lln1} \end{align}
with high probability. Setting $\delta = p/2$, we then have 
$$I_M^n := \{ 1\leq i \leq n \colon X_{\pi_i}^n \leq M\} \subseteq \{1,2,\hdots, n_\delta \}$$ 
with high probability.
Since the $X_{\pi_i}^n$ are order statistics, it follows that
$$\sum_{i=1}^{n_\delta} X_{\pi_i}^n \geq \sum_{i \in \pi^{-1} (I_M^n)} X_i = \sum_{i=1}^n X_i \ind{X_i \leq M}$$
with high probability.
Dividing by $n$ and applying \eqref{eq:lln1} gives
\begin{align}
\f 1n \sum_{i =1}^{n_\delta} X_{\pi_i}^n \geq \f 1n  \sum_{i=1}^n X_i \ind{ X_i \leq M} \geq  (1- \epsilon/2) \mu  \label{eq:ub}
\end{align}
with high probability. The relation at \eqref{eq:ub} is the first part of \eqref{eq:lln}. The second part follows from \eqref{eq:ub} along with the observation, again from the law of large numbers, that 
\begin{align}
\f 1n \sum_{i=1}^{n_\delta} X_{\pi_i}^n + \f 1n \sum_{i =n_\delta+1}^n X_{\pi_i}^n = \f 1n \sum_{i=1}^n X_i  \leq (1+ \epsilon/2) \mu \label{eq:end}
\end{align}
with high probability. 

Next, suppose that $M = \max \{ x \colon P(X_1 = x) >0\}$ exists for some finite $M$ so that the positive support of $X_1$ is bounded. Let $p = P(X_1 = M) >0$. Note that if $p=1$, then the $X_i$'s are deterministic and the desired claim is trivial. So, suppose that $0<p<1$ and set $\delta = \epsilon p \mu / M$. Since the mean $\mu$ of $X_1$ is bounded by the maximum value $M$ in the support of $X_1$ we have $\mu /M < 1$. This, along with the assumption $\epsilon <1$, implies that $\delta < p$. Thus, the law of large numbers ensures that $X_{\pi_i}^n = M$ for all $i > n_\delta$ with high probability. Using this along with the simple observation that $n_\delta \geq  (1- \delta)n $ gives
\begin{align}
\f 1n \sum_{i=n_\delta +1}^{n} X_{\pi_i}^n = \f 1 n (n - n_\delta  ) M  \leq   \delta  M = \epsilon p \mu \leq \epsilon \mu\label{eq:finM}
\end{align}
with high probability. 
This gives the second part of \eqref{eq:lln}, from which the first part can be derived using similar reasoning as at \eqref{eq:end} except with the equality
\begin{align}
\f 1n \sum_{i=1}^{n_\delta} X_{\pi_i}^n + \f 1n \sum_{i =n_\delta+1}^n X_{\pi_i}^n = \f 1n \sum_{i=1}^n X_i  \geq (1- \epsilon ) \mu \label{eq:end2}
\end{align}
with high probability. 
\end{proof}

\begin{lemma} \thlabel{lem:jn}
Let $j_n$ be as defined at \eqref{eq:mr-def}. There exists $\alpha>0$ such that $j_n \leq (1-\alpha)n$ with high probability.
\end{lemma}

\begin{proof}
Let $X_i = D_i(D_i -2)$. Set $E [X_1] = \beta >0$. \thref{lem:lln} ensures that there exists $\alpha >0$ such that 
$$ \sum_{i=1}^{\lceil (1-2\alpha) n \rceil} X_{\pi_i} \geq \f{\beta}{2} n >0$$
with high probability. 
Referring back to the definition of $j_n$ at \eqref{eq:mr-def}, this ensures that $j_n \leq (1-\alpha)n$ as claimed. 
\end{proof}

We will require a bound on $\EE_n$, the number of half-edges connected to $H_n^c$, defined at \eqref{eq:mr-def}. 

\begin{lemma} \thlabel{lem:En}
 Let $\EE_n$ be as defined at \eqref{eq:mr-def}. For all $\epsilon >0$, there exists $\lambda>0$ such that $\EE_n < \epsilon n$ with high probability. 
\end{lemma}

\begin{proof}
Consider the event $O_v$ from \eqref{eq:open}. Notice that for $v \neq 1$,  $T_v^r$ is distributed as the maximum of $D_v$ many independent exponential with mean $\lambda$ random variables, and $T_v^b$ is distributed as the minimum of $D_v$ many independent unit exponential random variables. These are elementary distributions, from which it is straightforward to derive the equality
\begin{align}
\P_\lambda(T_v^r < T_v^b) = \sum_{k=0}^\infty  \int_0^\infty (1- e^{-\lambda x})^k ke^{-kx} \:dx P(D_v = k).\label{eq:exp}
\end{align}
Taking $\lambda \to \infty$ and applying the dominated convergence theorem gives that the expression in \eqref{eq:exp} converges to $1$. A slight modification gives the same result for $v=1$.

Since $E D_1 <\infty$, \thref{lem:lln} ensures that there exists  $\delta>0$ such that 
\begin{align} \sum_{i=\lceil (1-\delta) n \rceil }^n D_{\pi_i} < \epsilon n\label{eq:lln-2}
\end{align}
with high probability. As explained after \eqref{eq:exp}, let $\lambda$ be such that the probability that a vertex is open is greater than $1-(\delta/2)$. Further observe that the events $\{O_v\}_{v \in G_n}$ are independent since they concern disjoint sets of the underlying independent red and blue passage times.  Thus, the number of non-open vertices is dominated by a $\Bin(n, \delta/2)$ random variable. The law of large numbers ensures that with high probability 
\begin{align}
|H_n^c| \leq  \delta  n. \label{eq:Hnc}
\end{align}
On this event and \eqref{eq:lln-2}, we then have
$$\EE_n = \sum_{v \in H_n^c} D_v \leq \sum_{i = \lceil (1- \delta) n \rceil }^n D_{\pi_i} < \epsilon n$$
with high probability. 
\end{proof}

\begin{lemma}\thlabel{lem:hatj}
Let $\hat D$ be as in \thref{lem:law} and define $\hat j_n$ analogously to how $j_n$ is defined at \eqref{eq:mr-def} with $\hat D_{\hat \pi_i}$ the $i$th order statistic of $\hat {\mathcal D} _n = (\hat D_1, \hdots , \hat D_n)$. For $\lambda$ sufficiently large, there exists $\hat \alpha>0$ such that $\hat j_n \leq (1 -\hat \alpha)n$ with high probability.
\end{lemma}

\begin{proof}
    The idea of the proof is to consider the damage, i.e., worst case minimizing effects, on $\sum_{i=1}^{j_n} D_{\pi_i} (D_{\pi_i}-2)$ after deleting up to $2\mathcal E_n$ (defined at \eqref{eq:mr-def}) half-edges to form the $\hat D_{\hat \pi_i}$. We then repair that damage by extending the range of the sum to include some larger degree terms. Such terms are available since \thref{lem:jn} ensures that $j_n \leq (1-\alpha) n$ for some $\alpha >0$ and  \thref{lem:En} allows us to control the number of half-edges of closed vertices $\mathcal E_n$. 
    
    To be more specific, the argument goes in four stages: Stage one shows that the damage from removing half-edges from  vertices $\pi_1,\hdots , \pi_{j_n}$ is at most a constant times $D_{\pi_{j_n}} \EE_n$. Stage two quantifies the repairing effect of including vertices $\pi_i$ for $(1-\alpha)n < i \leq (1-\alpha)n + C\EE_n$. If the first $(1-\alpha)n+C\EE_n$ vertices of $\hat \pi$ are the same as the first $(1-\alpha)n+C\EE_n$ vertices of $\pi$, that is, no vertex after $i> (1-\alpha)n+ C\EE_n$ in $\pi$ is permuted to the front in $\hat \pi$, then the first two stages are sufficient. In stage three we note that, the index of the first $(1-\alpha)n + C\EE_n$ vertices could be pushed back by at most $2\EE_n$. In the worst case, the additional vertices may have degrees $\hat D_i = 1$ (and so $\hat D_i(\hat D_i -2) = -1$). Lastly, we combine these observations and use the fact that $\EE_n$ can be made small to prove the claimed result. Now we provide the details. 

    It follows from \thref{lem:jn} that there exists $\alpha >0$ such that $j_n \leq (1-\alpha)n$ with high probability. As guaranteed by \thref{lem:En}, let $\lambda>0$ be such that $\EE_n < (\alpha/10) n $ with high probability. The total number of half-edges removed is at most $2\EE_n$ where the maximum is attained when each closed half-edge is attached to an open half-edge. We proceed by assuming the occurrence of the event $\{j_n<(1-\alpha)n\} \cap \{\mathcal E_n < (\alpha/10)n\}$, which has high probability.
    
    For each half-edge attached to vertices $\pi_1, \pi_2 \dots, \pi_{j_n}$, the reduction to the quantity
    \begin{equation} \label{eq:jn-sum}
        \sum_{i=1}^{j_n} D_{\pi_i}(D_{\pi_i}-2) >0
    \end{equation}
    by removing a single half-edge attached to a vertex with degree $d$ before the removal occurs is $d(d-2)-(d-1)(d-3) = 2d-3$. As the terms are ordered by size, the maximal reduction occurs when $d=D_{\pi_{j_n}}$. Since $d(d-2)$ is non-positive for $d=0,1,2$, $D_{\pi_{j_n}} \ge 3$. Since there are at most $2\EE_n$ edges removed, the sum \eqref{eq:jn-sum} is reduced by at most 
    $$ 2(2D_{\pi_{j_n}}-3) \EE_n < (4D_{\pi_{j_n}}-6)\frac \alpha {10} n.$$
    Then,
    \begin{equation}\label{eq:D-sum-1}
        \sum_{i=1}^{j_n}\hat D_{\pi_i}(\hat D_{\pi_i-2}) \ge \sum_{i=1}^{j_n} D_{\pi_i}( D_{\pi_i-2}) - (4D_{\pi_{j_n}}-6)\frac \alpha {10} n.
    \end{equation}
    
    We next consider the vertices between indices $(1-\alpha)n$ and $(1-(4\alpha/10))n$. As $\EE_n < (\alpha/10)n$, there are at most $2\EE_n<(2\alpha/10)n$ half-edges removed from these vertices. Then, there are at least $(6\alpha/10) n D_{\pi_{j_n}}  - (2\alpha/10)n$ half-edges attached to vertices between indices $(1-\alpha)n$ and $(1-(4\alpha/10))n$ after removal.
    
    We lower bound the contribution of vertices between indices $(1-\alpha)n$ and $(1-(4\alpha/10))n$ by constructing from scratch a degree sequence that minimizes the sum. Each additional half-edge contributes to the sum by $2d-3$ where $d$ is the number of the half-edges of the vertex after the addition. Since $2d-3$ is increasing,  the sum $\sum_{i=(1-\alpha)n}^{(1- 4\alpha/10)n} \hat D_{\pi_i}(\hat D_{\pi_i}-2)$ is least when every half-edge is added to the vertex of the least degree. Then, there are at least $(6\alpha/10) n D_{\pi_{j_n}}  - (2\alpha/10)n -2(6\alpha/10)n= (6D_{\pi_{j_n}}-14)(\alpha/10) n$ half-edges added as the third or higher edge of their vertex, each contributing at least $2\times 3 - 3 = 3$ to the sum. Then,
    \begin{equation} \label{eq:D-sum-2}
        \sum_{i=\lceil(1-\alpha)n\rceil}^{\lfloor(1- 4\alpha/10)n\rfloor} \hat D_{\pi_i}(\hat D_{\pi_i}-2) \ge (18D_{\pi_{j_n}}-42)\frac \alpha {10} n.
    \end{equation}
    
    Finally, to account for the scenario that the last $2\mathcal E_n$ edges suffer edge removal and become degree-$0$ edges, for example, we need to consider the case where the first $j_n$ edges are ``pushed back''. There are at most $2\EE_n< (2\alpha/10)n$ vertices after index $(1-(4\alpha/10)n)$ that can be permuted to an earlier index in $\hat \pi$ which could push the index of $\pi_i$ back by at most $(2\alpha/10)n$. In the worst case, all such vertices have degree $1$ after edge removal. Then, by the above, \eqref{eq:D-sum-1}, \eqref{eq:D-sum-2}, and $D_{\pi_{j_n}}\ge 3$, we have
    \begin{align*}
        &\sum_{i=1}^{\lfloor(1-(2\alpha/10)n)\rfloor} 
        \hat D_{\hat \pi_i} (\hat D_{\hat \pi_i}-2) \\
        &\ge \sum_{i=1}^{j_n} \hat D_{
        \hat \pi_i} (\hat D_{\hat \pi_i}-2) + \sum_{i=\lceil(1-\alpha)n\rceil}^{\lfloor(1-(4\alpha/10))n\rfloor} \hat D_{\pi_i} (\hat D_{\pi}-2) + (-1) (2\alpha/10)n \\
        &\ge \sum_{i=1}^{j_n} D_{\pi_i}( D_{\pi_i-2}) - (4D_{\pi_{j_n}}-6)\frac \alpha {10} n + (18D_{\pi_{j_n}}-42)\frac \alpha {10} n - 2 \frac \alpha {10} n \\
        &\ge \sum_{i=1}^{j_n} D_{\pi_i}( D_{\pi_i-2}) + (14 D_{\pi_{j_n}}- 38) \frac \alpha {10}n >0.
    \end{align*}
    Letting $\hat \alpha = \alpha/5$, it follows that $\hat j_n \le (1-\hat \alpha)n$ with high probability.
\end{proof}

\begin{proposition}\thlabel{prop:Hn}
There exist $\lambda,\delta>0$ such that \eqref{eq:Hgiant} holds. 
\end{proposition}

\begin{proof}
    By \thref{lem:law}, the subgraph $H_n$ is sampled from the configuration model with degree sequence $\hat{\mathcal D}_n$. It follows from \cite{janson2009probability} that 
    $$\liminf_{n \geq 1} \P_\lambda(G_n \text{ is simple})  >0.$$
    Since the $G_n$ are sampled independently of one another, there is almost surely a random increasing subsequence $n_1, n_2,\hdots$ such that  for $k\geq 1$ each $G_{n_k}$ is simple. As $H_{n_k}$ is a subgraph, the $H_{n_k}$ are also simple. 

    \thref{lem:hatj} implies that for $\lambda$ large enough there exists $\hat \alpha >0 $ with $\hat j_{n_k} < (1-\hat \alpha)n_k$ with high probability. 
    Since $E [D_i] = \mu < \infty$, the law of large numbers ensures that $\M_{n_k} \leq 2\mu n_k$ with high probability. Together, these observations imply 
    \begin{align}
    \hat \S_{n_k} \geq \sum_{ i= \lceil (1-\hat \alpha)n_k\rceil}^{n_k}\hat D_{\hat \pi_i} \geq 3 \hat\alpha n_k = \f {3\hat \alpha}{2 \mu} 2 \mu  n_k  \geq  \f {3\hat \alpha}{2 \mu}  \M_{n_k}\label{eq:og}
    \end{align}
    with high probability. The second inequality is true because there are at most $\hat \alpha n_k+1$ indices and $\hat D_{\hat \pi_i}\ge 3$ for all $i\ge \hat j_{n_k}$. Setting $\epsilon = 3\hat\alpha / (2 \mu)$ in \thref{lem:giant}, we conclude that for $\lambda$ large enough there exists $\delta>0$ such that $|H_{n_k}^{(1)}| > \delta n_k$ with high probability (as $k \to \infty$).
    Since \eqref{eq:Hgiant} is defined with a $\limsup$, a result for the subsequence is enough to infer \eqref{eq:Hgiant}. 
\end{proof}

\begin{proof}[Proof of \thref{thm:main}]
That $\lambda_c(\G)<\infty$ follows immediately from \thref{lem:giant} and \thref{prop:Hn}. That $\lambda_c(\G) \geq \Lambda$ follows from  \eqref{eq:finE}, which is proven in the next section.
\end{proof}





\section{Proof of \thref{thm:ER}}
\label{sec:ER}

\begin{proof}[Proof of the upper bound \eqref{eq:finE}]

We sample a graph $G_n$ from the configuration model with degree sequence $\D_n$. Let $\Gamma_k$ be the set of all vertex self-avoiding paths of length $k$ starting at 1 that are present in $G_n$. Interpret $\Gamma_0 = \{1\}$ as the path of length 0 starting at 1.  We say that red \textit{survives} on a path $\gamma \in \Gamma_k$ if, for chase-escape restricted only to the passage times along $\gamma$, the terminal vertex of $\gamma$ is eventually colored red. We emphasize that survival along $\gamma$ ignores the influence of red and blue from all edges not belonging to $\gamma$ and only depends on the passage times along $\gamma$. 
 
Let $A_k = A_k(\lambda)$ be the event that $k$ is ever colored red in chase-escape on the infinite path $0,1,2,\hdots$ with $0$ initially blue and $1$ initially red. Observe that $$\P_\lambda(\text{red survives on a path $\gamma$ of length $k$}) = \P_\lambda(A_k).$$ By \cite[Lemma 2.2]{DJT}, for $\lambda < 1$ and all $k \geq 1$.
\begin{equation}
\P_\lambda (A_k) \leq C_\lambda \left( \frac{4 \lambda}{(1+ \lambda )^2}\right)^kk^{-3/2} \label{eq:Ak}
\end{equation}
with 
\begin{align}
    C_\lambda = \sum_{i=0}^\infty (2i +1) \lambda ^i = \f {1+\lambda} { (1- \lambda)^2}. \label{eq:Cl}
\end{align}

For any vertex $v \in \R_n$, it is required that there is a path red survives on with $v$ the terminal point. Hence,
\begin{equation}
\label{3}
|\R_n| \leq  \sum^n_{k=0} \sum_{\gamma \in \Gamma_k} \ind{\text{red survives on } \gamma}.
\end{equation}
Taking expectation and using the fact that $\Gamma_k$ is independent of the identically distributed $\ind{\text{red survives on $\gamma$}}$ gives
\begin{align}
    \E_\lambda [|\R_n|] \leq \sum_{k=0} \E_\lambda [|\Gamma_k|] \P_\lambda(A_k). \label{eq:ER}
\end{align}

The quantity $\P_\lambda(A_k)$ is bounded at \eqref{eq:Ak}. A standard branching process construction (see \cite{bollobas2015old} for example) shows that the total number of paths of length $k$ is dominated by a branching process in which the root has $D_1$ children and subsequent generations have offspring distribution $D^*-1$ where $D^*$ is the size-biased distribution of $D$. Namely, $P(D^* = i) = i P(D=i)/ E [D]$ for $i \geq 1$. One easily checks that $E [D^*] = E [D^2] / E [D] = a+1$ with $a$ defined at \eqref{eq:Lambda}. It follows that 
$\E_\lambda [|\Gamma_k|]\leq E[D] a^{k-1}.$ Note that by the condition at \eqref{eq:mr}, we must have $a >1$.

Applying these bounds to \eqref{eq:ER} yields
\begin{align}
    \E_\lambda [|\R_n|]& \leq 1+  \sum_{k=1}^n E[D] a^{k-1}  C_\lambda \left( \frac{4 \lambda}{(1+ \lambda )^2}\right)^kk^{-3/2} \\
    &\leq 1 + \f{E[D] C_\lambda}a \sum_{k=1}^\infty  \left( \frac{4a \lambda}{(1+ \lambda )^2}\right)^kk^{-3/2}\label{eq:ER2}.
\end{align}
The value of $\Lambda$ at \eqref{eq:Lambda} is the solution to $a\Lambda / (1+\Lambda)^2=1$ for which $\lambda \leq \Lambda$ implies
\begin{align}
\left( \frac{4 a \lambda}{(1+ \lambda )^2}\right)\leq 1.\label{eq:l<L}
\end{align}

Applying the bound at \eqref{eq:l<L}, the easily proven inequality $\sum_{k=1}^\infty k^{-3/2} \leq  3$, and  the formula at \eqref{eq:Cl} for $C_\lambda$ to \eqref{eq:ER2} gives
\begin{align}
    \E_\lambda [| \R_n|] \leq 1 + \f{ 3 E [D]} a \left(\f {1 + \lambda }{(1-\lambda)^2} \right)\ &\leq  1 + \f{ 3 E [D]} a \left(\f {1 + \Lambda }{(1-\Lambda)^2} \right) \\
    &= 1 + \f{ 3E[D]^2}{E[D^2] -E[D]}\left(\f {1 + \Lambda }{(1-\Lambda)^2} \right).
\end{align}
Setting 
$$C=1 + \f{ 3E[D]^2}{E[D^2] -E[D]}\left(\f {1 + \Lambda }{(1-\Lambda)^2} \right)$$
gives \eqref{eq:finE}. Note that $C<\infty$ since $a>1$ implies that $\Lambda <1$. 
\end{proof}

\begin{proof}[Proof of the lower bound \eqref{eq:ERl}]
    Let $\mathcal T$ be a random tree in which the root has a $D$-distributed number of children and all other vertices have an independent $(D^*-1)$-distributed number of vertices with $D^*$ the size-biased version of $D$. So, it is equivalent to write $a$ from \eqref{eq:Lambda} as $a = E [D^*-1]$. 
    
    Suppose first that $a< \infty$. Let $\mathcal T_m$ be the first $m$ generations of the tree. It is proven in \cite[Lemma 4]{bollobas2015old} that for any fixed $m$ the subgraph $\mathbb B_n(1,m)$ of vertices within distance $m$ of $1$ in $G_n$ may be coupled to equal $\mathcal T_n$ with high probability. We will denote the coupling probability measure by $\PP$. It follows from \cite[Theorem 1.1]{bordenave2014extinction} that, for chase-escape on $\mathcal T$, red reaches infinitely many vertices with positive probability so long as $\lambda > \Lambda$. In particular, the probability that red reaches a vertex at distance $m$ is bounded below by some constant $\beta >0$.     
    Letting $C_m$ be the event that $\mathcal T_m$ and $\mathbb B_n(1,m)$ can be coupled, we then have $$\liminf_{n >0} \E_\lambda[|\R_n|] \geq \liminf_{n >0} \E_\lambda[|\R_n| \mid C_m] \PP(C_m) \geq \beta m \PP(C_m).$$  As $\beta >0$ does not depend on $m$ and $\PP(C_m) \to 1$ as $n \to \infty$ for any fixed $m$, it follows that $\liminf_{n >0} \E [|\R_n|] = \infty$. 
    
    If $a=\infty$, then we fix $L>0$ and let $\mathcal T^L$ be a randomly sampled embedded tree with truncated offspring distribution $(D^* -1) \wedge L$. Let 
    $$b_L =\limsup_{k >0} |V_k(\mathcal T^L)|^{1/k}$$
    be the branching number of $\mathcal T^L$. Since $a = \infty$ and $D^*$ is almost surely finite, we have $b_L = \E [(D^* -1) \wedge L]\to \infty$ as $L \to \infty$. 
    Thus, for any $\lambda >0$, we can choose $L$ large enough so that $\Lambda((D^*-1) \wedge L) < \lambda$, then apply \cite[Theorem 1.1]{bordenave2014extinction} and similar reasoning as the $a<\infty$ case to deduce that $\E_\lambda[|\R_n|] \to \infty$. 
\end{proof}


\bibliographystyle{amsplain} 
\bibliography{references_updated}



\end{document}